\documentclass[12pt]{article}
\usepackage{fullpage,amsmath,amssymb,amsthm,hyperref}

\everymath{\displaystyle}
\newcommand{\FF}{F}
\newcommand{\GG}{G}
\newcommand{\Split}{S}
\newcommand{\Perm}{\mathfrak{S}}

\newcommand{\BB}{B}
\newcommand{\UU}{U}
\newcommand{\UK}{U^K}
\newcommand{\US}{U^S}
\newcommand{\Uamb}{U^\textup{amb}}
\newcommand{\CS}{\mathfrak{c}S}
\newcommand{\CUK}{\mathfrak{c}\UK}
\newcommand{\CUS}{\mathfrak{c}\US}
\newcommand{\CUamb}{\mathfrak{c}\Uamb}
\newcommand{\CB}{\mathfrak{c}\BB}
\newcommand{\BC}{\textup{BC}}
\newcommand{\BCiso}{\textup{BC}^*}

\theoremstyle{plain}
\newtheorem{theorem}{Theorem}[section]
\newtheorem{lemma}[theorem]{Lemma}
\newtheorem{corollary}[theorem]{Corollary}
\newtheorem{proposition}[theorem]{Proposition}

\theoremstyle{definition}
\newtheorem{definition}[theorem]{Definition}

\title{Split graphs: combinatorial species and asymptotics}
\author{\vspace{0pc} Justin M. Troyka\thanks{Work done during graduate fellowship, Department of Mathematics, Dartmouth College, Hanover, NH, United States.} \vspace{.25pc} \\
\normalsize \vspace{0pc} Department of Mathematics and Statistics \vspace{-.25pc} \\
\normalsize \vspace{0pc} York University \vspace{-.25pc} \\ 
\normalsize Toronto, ON, Canada \\
\normalsize \texttt{jmtroyka@yorku.ca}}
\date{2019 May 12\vspace{-1pc}}

\begin{document}

\maketitle

\begin{abstract}
A split graph is a graph whose vertices can be partitioned into a clique and a stable set. We investigate the combinatorial species of split graphs, providing species-theoretic generalizations of enumerative results due to B\'ina and P\v{r}ibil (2015), Cheng, Collins, and Trenk (2016), and Collins and Trenk (2018). In both the labeled and unlabeled cases, we give asymptotic results on the number of split graphs, of unbalanced split graphs, and of bicolored graphs, including proving the conjecture of Cheng, Collins, and Trenk (2016) that almost all split graphs are balanced.\end{abstract}

A \emph{split graph} is a graph whose vertices can be partitioned into two sets, $K$ and $S$, such that the vertices in $K$ form a clique (complete subgraph) and the vertices in $S$ form a stable set (independent set). A partition of a graph in this way is called a \emph{$KS$-partition}.

Split graphs are a well-known class of perfect graphs, and they are precisely the graphs $G$ such that both $G$ and $\overline{G}$ are triangulated. A summary of split graphs from that perspective is found in Golumbic \cite[Sec.\ 6]{Golumbic}. Moreover, Hammer and Simeone \cite{HS} characterize split graphs in terms of their degree sequences and thereby provide an efficient algorithm for determining whether a graph is a split graph.

For a given class of graphs, two enumerative problems are to count the unlabeled graphs (the isomorphism classes of graphs) on $n$ vertices and to count the labeled graphs on $n$ vertices (in which the $n$ vertices have distinct labels $1$ through $n$). In recent years, there has been interest in counting both unlabeled and labeled split graphs. Royle \cite{Royle} gives a bijection between unlabeled split graphs and minimal set covers, of which the latter had previously been counted. Cheng, Collins, and Trenk \cite{CCT} explore connections between split graphs and Nordhaus--Gaddum graphs, and Collins and Trenk \cite{CollinsTrenk} give bijections between unlabeled split graphs, $XY$-graphs (called bicolored graphs in this paper), and bipartite posets. Collins and Trenk also characterize the minimal set covers, $XY$-graphs and bipartite posets that correspond to unbalanced split graphs (defined in Section \ref{sec:SplitGraphs}). As for counting labeled split graphs, the exact enumeration has been done by B\'ina and P\v{r}ibil \cite{BP}; but our results yield a formula (Corollary \ref{cor:Bina}) somewhat simpler than theirs. The asymptotic enumeration of labeled split graphs was done much earlier, by Bender, Richmond, and Wormald \cite{BRW}, who show the asymptotic number is the same as for a few other classes of labeled graphs, including what our paper calls bicolored graphs.

In this paper, we extend some of these enumerative results to the setting of combinatorial species. The theory of combinatorial species, introduced by Joyal \cite{Joyal}, is a powerful conceptual framework for thinking about combinatorial structures that can be labeled or unlabeled, such as graphs. Our results on the level of species are valuable not only because we can instantly recover from them the known labeled and unlabeled enumerations, but because they say something more than enumeration alone: for two types of structures to have the same (or isomorphic) species means that they are combinatorially equivalent in some sense. For a comprehensive treatment of species theory, see \cite{BLL}; for a summary of the theory and some applications to graph enumeration, see \cite{comps}.

Section \ref{sec:species} is an exposition of the parts of species theory we will use. In Section \ref{sec:SplitGraphs} we provide more background on split graphs and prove several identities about the species of split graphs and related species. This culminates in Theorem \ref{thm:main}, a species version of the result from \cite{CCT, CollinsTrenk} that the number of unbalanced split graphs on $n$ vertices equals the number of split graphs on $\le n-1$ vertices. In Section \ref{sec:BicoloredAsymptotics} we relate the species of split graphs to the species of bicolored graphs, and we prove several asymptotic results, including Theorem \ref{thm:almost-all-balanced-unlabeled}, that almost all split graphs are balanced, which was a conjecture of Cheng, Collins, and Trenk \cite{CCT}.

\section{Combinatorial species} \label{sec:species}

This section is an all-too-brief statement of some of the ideas originating from Joyal \cite{Joyal}; for many more details, see \cite{Joyal, BLL, comps}. We describe two equivalent ways of viewing a combinatorial species: one in terms of categories and functors, and one in terms of permutation group actions.

A \emph{(combinatorial) species} $\FF$ is a functor from the category of finite sets with bijections to itself. That is, $\FF$ is a rule that does the following:
\begin{itemize}
\item To each finite set $I$, assigns a finite set of \emph{structures}, denoted $\FF[I]$; \item To each bijection $\varphi$ between finite sets $I$ and $J$, assigns a bijection $\FF[\varphi]$ between the sets $\FF[I]$ and $\FF[J]$.
\end{itemize}
The set $I$ is thought of as a set of labels, in which case $\FF[I]$ is the set of $\FF$-structures in which each label in $I$ occurs exactly once, and the bijection $\FF[\varphi]$ maps each $\FF$-structure on label set $I$ to an $\FF$-structure on label set $J$ obtained by replacing label $i$ with label $\varphi(i)$ for all $i \in I$.

A basic example of a species is the species of sets, denoted $E$ (for \textit{ensemble}, the French word for ``set''). This is defined by $E[I] = \{I\}$: there is one structure with label set $I$, namely the set $I$ itself. For a bijection $\varphi\colon I \to J$, we must have $E[\varphi]$ map the one element of $\{I\}$ to the one element of $\{J\}$.

For the species in this paper, the structures are graphs with a certain property. In this case, $\FF[I]$ is the set of graphs with that property whose vertices are labeled with the elements of $I$, and $\FF[\varphi] \colon \FF[I] \to \FF[J]$ maps each graph in $\FF[I]$ to the corresponding isomorphic graph in $\FF[J]$.

In the special case of $I = J = [n]$, we write $\FF[n]$ instead of $\FF[{[n]}]$, and $\FF[n]$ is considered as the set of labeled $\FF$-structures of size $n$. A bijection $[n] \to [n]$ is a permutation in the symmetric group $\Perm_n$. Given $\sigma \in \Perm_n$, the function $\FF[\sigma]$ is a permutation of the set $\FF[n]$. Thus $F$ induces an action of $\Perm_n$ on $\FF[n]$ for each $n$.

If $\FF$ is a species of graphs with a certain property, then $\Perm_n$ acts on $\FF[n]$ by graph isomorphisms. Given a graph $g \in \FF[n]$, each permutation induces an isomorphism from $g$ to some graph in $\FF[n]$, and the permutations that map $g$ to itself are the automorphisms of $g$.

In the example of $E$ (the species of sets), since there is only one $E$-structure on label set $[n]$, the symmetric group must act trivially on it. Thus the automorphism group of an $E$-structure (a set) of size $n$ is all of $\Perm_n$: permuting the elements of a set does not change what the set is.

\subsection{Labeled and unlabeled structures, generating functions, and species isomorphism}

Let $\FF$ be a species. The elements of $\FF[n]$ are the (labeled) $\FF$-structures of size $n$. The orbits of $\FF[n]$ under the action of $\Perm_n$ are isomorphism classes of $\FF$-structures, and we write $\FF[n]/\Perm_n$ to denote the set of these orbits. We think of each orbit as an \emph{unlabeled} $\FF$-structure. For example, in the case where $\FF$ is a species of graphs, the orbits are the isomorphism classes of graphs, which are unlabeled graphs.

A species $\FF$ has three kinds of generating functions associated with it:
\begin{itemize}
\item $\FF(x)$ denotes the exponential generating function for the labeled $\FF$-structures:
\[ \FF(x) = \sum_{n\ge0} |\FF[n]|\,\frac{x^n}{n!}. \]
\item $\widetilde{\FF}(x)$ denotes the ordinary generating function for the unlabeled $\FF$-structures:
\[ \widetilde{\FF}(x) = \sum_{n\ge0} |\FF[n]/\Perm_n|\,x^n. \]
\item $Z_\FF(p_1,p_2,\ldots)$ denotes the cycle index series of $\FF$, a generating function in infinitely many variables. We will not use the cycle index series in this paper, but it is important because it generalizes $\FF(x)$ and $\widetilde{\FF}(x)$, in the sense that $Z_\FF(x, 0, 0, \ldots) = \FF(x)$ and $Z_\FF(x, x^2, x^3, \ldots) = \widetilde{\FF}(x)$.
\end{itemize}
Note that for two species $\FF$ and $\GG$ we can have $\FF(x) = \GG(x)$ without having $\widetilde{\FF}(x) = \widetilde{\GG}(x)$, or the other way around.

For the set species $E$, since $|E[n]| = 1$ for all $n$, we have $E(x) = \sum_{n\ge0} \frac{x^n}{n!} = e^x$ and $\widetilde{E}(x) = \sum_{n\ge0} x^n = \frac{1}{1-x}$.

A \emph{species isomorphism} $\alpha$ from $\FF$ to $\GG$ is a natural equivalence from $\FF$ to $\GG$ as functors. That is, $\alpha$ is a family of bijections $\alpha_I \colon \FF[I] \to \GG[I]$ for each label set $I$ that commutes with bijections between label sets: for any bijection $\varphi\colon I \to J$, we have $\alpha_J \circ \FF[\varphi] = \GG[\varphi] \circ \alpha_I$. Viewing species in terms of group actions, this is equivalent to a bijection $\alpha_n \colon \FF[n] \to \GG[n]$ for each $n$ that preserves the action of $\Perm_n$, meaning that $\alpha_n(\sigma \cdot f) = \sigma \cdot \alpha_n(f)$ for all $\sigma \in \Perm_n$ and $f \in \FF[n]$. Thus, $\FF$ and $\GG$ are isomorphic if as functors they are naturally equivalent, or if $\FF[n]$ and $\GG[n]$ are isomorphic $\Perm_n$-sets for every $n$; in this case we simply write $\FF = \GG$.

The species $E$ is isomorphic to the species of complete graphs (cliques), and it is also isomorphic to the species of edgeless graphs (stable sets). It is useful to think of $E$ in this way when we think about building certain species of graphs from other species of graphs.

If $\FF$ and $\GG$ are isomorphic, then $Z_\FF = Z_\GG$; and, as already discussed, if $Z_\FF = Z_\GG$, then $\FF(x) = \GG(x)$ and $\widetilde{\FF}(x) = \widetilde{\GG}(x)$. The converse does not hold for any of these implications. Thus, species isomorphism is the finest notion of equality between combinatorial structures that we have discussed here. We will view isomorphic species as being equal, simply writing $F = G$ if $F$ and $G$ are isomorphic.

\subsection{Addition and multiplication of species}

Given species $\FF$ and $\GG$, the sum $\FF + \GG$ is a species whose structures are $\FF$-structures or $\GG$-structures --- that is, $(\FF+\GG)[I] = \FF[I] \sqcup \GG[I]$ (the disjoint union). From the perspective of group actions, $(\FF+\GG)[n]$ is the disjoint union of $\FF[n]$ and $\GG[n]$ as $\Perm_n$-sets.

If $\FF$ and $\GG$ are species of two different types of graphs, then $\FF+\GG$ is the species of graphs of one type or the other type, provided that no graph is of both types. If the two types of graphs do overlap, then $\FF+\GG$ double-counts the intersection.

The product $\FF \cdot \GG$ is defined as the species whose structures are ordered pairs of an $\FF$-structure and a $\GG$-structure. That is, the structures in $(\FF \cdot \GG)[I]$ are obtained by partitioning the labels as $I = U \cup V$ and forming ordered pairs $(\mathrm{f}, \mathrm{g})$ with $\mathrm{f} \in \FF[U]$ and $\mathrm{g} \in \GG[V]$. The size of such an ordered pair as an $(\FF \cdot \GG)$-structure is the sum of the sizes of its two components.

The sum and product of species correspond to the sum and product of their generating functions: that is, $(\FF+\GG)(x) = \FF(x) + \GG(x)$ and $(\widetilde{\FF+\GG})(x) = \widetilde{\FF}(x) + \widetilde{\GG}(x)$ and $Z_{\FF+\GG} = Z_\FF + Z_\GG$, and likewise for multiplication. Furthermore, sum and product respect species isomorphism, in the sense that the isomorphism class of $\FF+\GG$ or $\FF \cdot \GG$ is determined by the isomorphism classes of $\FF$ and $\GG$.

We can use these operations to build up complicated species from simpler ones. For instance, $E \cdot E$ is the species of partitions of a set into an ordered pair of two sets, which are equivalent to subsets of a set. The labeled generating function is $E(x)\,E(x) = e^x\,e^x = e^{2x}$, in which the coefficient of $x^n/n!$ is $2^n$; hence a set of size $n$ has $2^n$ subsets. The unlabeled generating function is $\widetilde{E}(x)\,\widetilde{E}(x) = \frac{1}{(1-x)^2}$, in which the coefficient of $x^n$ is $n+1$; hence, if the elements of a set of size $n$ are unlabeled, then the set has $n+1$ distinguishable subsets (one of each size).

\subsection{Other species notation}

The zero species, denoted $0$, is defined as the species with no structures: $0[I] = \varnothing$ for every set $I$. The zero species is the additive identity: $\FF + 0 = \FF$ for all $\FF$. It also satisfies $0 \cdot \FF = 0$ for all $\FF$. The one species, denoted $1$, is defined as the species with one structure of size $0$ and no other structures: $1[\varnothing] = \{\varepsilon\}$ (a null structure) and $1[I] = \varnothing$ for every non-empty $I$. The one species is the multiplicative identity: $1 \cdot \FF = \FF$ for any species $\FF$.

Given a species $\FF$, the species $\FF_k$ is the species of $\FF$-structures of size $k$; that is, $\FF_k[I] = \varnothing$ if $|I| \not=k$, and $\FF_k[I] = \FF[I]$ if $|I| = k$. This means that $\textstyle\FF = \sum_{n\ge0} \FF_n$. We also write $\FF_{\le k}$ to denote the species of $\FF$-structures of size $\le k$, so $\FF_{\le k} = \FF_0 + \cdots + \FF_k$.

\subsection{The ring of virtual species}

We have defined addition and multiplication of species. By writing formal differences of species, such as $\FF - \GG$, we can extend the set of species to a ring, called \emph{the ring of virtual species}, in which the species $0$ and $1$ are the additive and multiplicative identities. The elements of this ring are called \emph{virtual species}. Every virtual species $\FF$ has the form $\FF^+ - \FF^-$, where $\FF^+$ and $\FF^-$ are ordinary species. If $\FF^+ - \FF^-$ and $\GG^+ - \GG^-$ are two virtual species in this form, then we identify them as the same virtual species if the two ordinary species $\FF^+ + \GG^-$ and $\GG^+ + \FF^-$ are isomorphic, i.e.\ $\FF^+ + \GG^- = \GG^+ + \FF^-$. The fact that virtual species form a ring makes algebraic manipulation much easier with virtual species than with combinatorial species alone, as we will see in our computations with the species of split graphs.

A virtual species $\FF$ is a unit if and only if $\FF_0 = 1$ or $\FF_0 = -1$, i.e.\ the ``constant term'' of $\FF$ is equivalent as a virtual species to the species $1$ or its negative $-1$. When this is the case, we will write $1/\FF$ to denote the multiplicative inverse of $\FF$, and we will write fractions accordingly.

\section{Split graphs} \label{sec:SplitGraphs}

We begin with a characterization of the sizes of $K$ and $S$ in a $KS$-partition of a split graph.

\begin{proposition}[{\cite[Thm.\ 6.2]{Golumbic}}] \label{prop:intro}
Let $G$ be a split graph and fix a $KS$-partition of $G$. Exactly one of the following holds:
\begin{itemize}
\item[(i)] $|K| = \omega(G)$ and $|S| = \alpha(G)$ and $G$ has a unique $KS$-partition;
\item[(ii)] $|K| = \omega(G) - 1$ and $|S| = \alpha(G)$ and there is $x \in S$ such that $K \cup \{x\}$ is a clique;
\item[(iii)] $|K| = \omega(G)$ and $|S| = \alpha(G) - 1$ and there is $x \in K$ such that $S \cup \{x\}$ is a stable set.
\end{itemize}
\end{proposition}

This proposition prompts the following definition, following \cite{CCT, CollinsTrenk}:

\begin{definition}
Let $G$ be a split graph. We say $G$ is \emph{balanced} if it has a unique $KS$-partition, and \emph{unbalanced} otherwise. A given $KS$-partition of $G$ is \emph{$S$-max} if $S$ is as large as possible, and \emph{$K$-max} if $K$ is as large as possible. Furthermore, a vertex $x$ as in case (ii) or (iii) of Proposition \ref{prop:intro} is called a \emph{swing vertex} of $G$.
\end{definition}

In this language, Proposition \ref{prop:intro} gives us these facts: every $KS$-partition of $G$ is $S$-max or $K$-max; if $G$ is unbalanced then the $K$ in a $K$-max partition has one vertex more than the $K$ in an $S$-max partition, and similarly for $S$; and a split graph is unbalanced if and only if it has a swing vertex.

In this section, we define four types of split graphs and define colored split graphs (Section \ref{sec:classifying}), we describe the species of split graphs and various related species (Section \ref{sec:species-split-related}), and we prove our main theorem and other identities on these species (Sections \ref{sec:identities} and \ref{sec:maintheorem}).

\subsection{Classifying split graphs} \label{sec:classifying}

A result due to Cheng, Collins, and Trenk \cite{CCT} describes the structure of the swing vertices of a split graph and lists all of the $KS$-partitions.

\begin{proposition}[{\cite[Thm.\ 10]{CCT}}] \label{prop:CCT}
Let $G$ be an unbalanced split graph, and let $A$ be the set of swing vertices of $G$. Then $A$ is either a clique or a stable set, and the non-swing vertices admit a partition into sets $Y$ and $Z$ such that every vertex in $A$ is adjacent to every vertex in $Y$ and no vertex in $Z$. Furthermore:
\begin{itemize}
\item If $A$ is a clique, then there is a unique $K$-max partition, namely $K = A \cup Y$ and $S = Z$; and the $S$-max partitions are given by $K = (A \smallsetminus \{a\}) \cup Y$ and $S = Z \cup \{a\}$ for $a \in A$.
\item If $A$ is a stable set, then there is a unique $S$-max partition, namely $K = Y$ and $S = A \cup Z$; and the $K$-max partitions are given by $K = Y \cup \{a\}$ and $S = (A \smallsetminus \{a\}) \cup Z$ for $a \in A$.
\end{itemize}
\end{proposition}

In particular, every vertex that is not a swing vertex is either in $K$ for all $KS$-partitions or in $S$ for all $KS$-partitions. Proposition \ref{prop:CCT} allows us to classify split graphs according to whether their swing vertices form a clique or a stable set:

\begin{definition}
Let $G$ be a split graph.
\begin{itemize}
\item $G$ is \emph{$K$-canonical} if the set of swing vertices forms a clique of size $\ge 2$;
\item $G$ is \emph{$S$-canonical} if the set of swing vertices forms a stable set of size $\ge 2$;
\item $G$ is \emph{ambiguous} if there is exactly one swing vertex;
\item $G$ is \emph{balanced} if there are no swing vertices.
\end{itemize}
\end{definition}

By Proposition \ref{prop:CCT}, every split graph is exactly one of those four types. In \cite{CCT}, a split graph that is $K$-canonical or ambiguous is an \emph{NG-1 graph}, and a split graph that is $S$-canonical or ambiguous is an \emph{NG-2 graph}. Our choice of names is because we will use the following ``canonical'' $KS$-partition of a split graph $G$: if $G$ is $K$-canonical, the canonical partition is the unique $K$-max partition; if $G$ is $S$-canonical, the canonical partition is the unique $S$-max partition; if $G$ is ambiguous, there is no canonical partition; and if $G$ is balanced, the canonical partition is the unique $KS$-partition.

\begin{definition} \label{defn:colored}
A \emph{colored split graph} is a split graph with a chosen $S$-max partition. Equivalently, it is a split graph with vertices colored green (Kelly green) and red (Scarlet) such that the green set and the red set are respectively $K$ and $S$ in an $S$-max partition.
\end{definition}

The four types of split graphs extend to colored split graphs. If a split graph is $S$-canonical, ambiguous, or balanced, then it has a unique $S$-max partition, so there is only one way to color the vertices to obtain a colored split graph. However, if a split graph is $K$-canonical, then there is more than one $S$-max partition, and each one gives rise to a different colored split graph.

\subsection{The species of split graphs and related species} \label{sec:species-split-related}

Let $\Split$ be the species of split graphs. What this means is that, for $I$ a finite set, $\Split[I]$ is the set of split graphs on vertex set $I$, and for any bijection $\varphi$ between finite sets $I$ and $J$, $\Split[\varphi]$ is the bijection between $\Split[I]$ and $\Split[J]$ that maps each split graph on $I$ to the isomorphic copy obtained by replacing label $i$ with $\varphi(i)$.

Also let $\BB$ be the species of balanced split graphs, and let $\UU$ be the species of unbalanced split graphs; note that $\Split = \BB + \UU$. Recall that, for a species $\FF$, $\FF(x)$ is the exponential generating function that counts labeled $\FF$-structures, and $\widetilde{\FF}(x)$ is the ordinary generating function that counts unlabeled $\FF$-structures. In \cite{CCT, CollinsTrenk}, it is proved that the number of unlabeled unbalanced split graphs on $n$ vertices equals the number of unlabeled split graphs on $\le n-1$ vertices: in the language of generating functions,
\begin{equation} \label{eqn:identity} \widetilde{\UU}(x) = \frac{x}{1-x} \widetilde{\Split}(x). \end{equation}
The idea behind this is as follows: every \emph{unlabeled} unbalanced split graph has a unique $K$-max partition, and from this partition we can obtain a split graph by removing all the swing vertices from $K$. However, this does not work for \emph{labeled} graphs, because a labeled split graph can have more than one $K$-max partition. So we will need to be more careful in order to obtain an identity like (\ref{eqn:identity}) for the labeled generating functions or on the level of species.

Let $\UK$ be the species of $K$-canonical split graphs. Let $\US$ be the species of $S$-canonical split graphs. Let $\Uamb$ be the species of ambiguous split graphs. Then $\UU = \UK + \US + \Uamb$. Taking the graph complement gives a bijection between $K$-canonical split graphs and $S$-canonical split graphs, and this bijection commutes with graph isomorphisms, so it is a species isomorphism between $\UK$ and $\US$; thus, $\UK = \US$.

Let $\CS$ denote the species of colored split graphs. The isomorphisms between colored graphs are the graph isomorphisms that preserve color. It is a consequence of Proposition \ref{prop:CCT} that a split graph has a unique $S$-max partition up to relabeling of the vertices. This means that colored split graphs with the same underlying graph are isomorphic. Therefore:

\begin{proposition} \label{prop:CS-Split}
$\widetilde{\CS}(x) = \widetilde{\Split}(x)$; in words, the number of unlabeled colored split graphs on $n$ vertices equals the number of unlabeled split graphs on $n$ vertices. \hfill $\square$
\end{proposition}

Let $\CUK$, $\CUS$, $\CUamb$, and $\CB$ denote the species of colored split graphs that are respectively $K$-canonical, $S$-canonical, ambiguous, and balanced. Then $\CS = \CUK + \CUS + \CUamb + \CB$. Since a split graph that is not $K$-canonical has only one colored split graph associated with it, we have the following equalities of species:

\begin{proposition} \label{prop:colored-equalities}
$\CUS = \US$ and $\CUamb = \Uamb$ and $\CB = \BB$, and $\CS - \CUK = \Split - \UK$.
\end{proposition}

\begin{proof}
The first three equalities are immediate from the remark preceding the proposition. The last one follows from the first three because $\Split = \UK + \US + \Uamb + \BB$ and $\CS = \CUK + \CUS + \CUamb + \CB$.
\end{proof}

Thus we will make no further use of the symbols $\CUS$, $\CUamb$, and $\CB$.

We summarize the species we have defined in this section and the relations between them that we have seen so far:
\[
\begin{tabular}{ll|ll}
$\Split$ & split graphs & $\CS$ & colored split graphs \\
$\UU$ & unbalanced split graphs & & \\
$\UK$ & $K$-canonical split graphs & $\CUK$ & $K$-canonical colored split graphs \\
$\US$ & $S$-canonical split graphs & & \\
$\Uamb$ & ambiguous split graphs & & \\
$\BB$ & balanced split graphs & &
\end{tabular}
\]
\begin{align*}
\Split &= \UU + \BB \\
\UU &= \UK + \US + \Uamb \\
\US &= \UK \\
\CS &= \CUK + \US + \Uamb + \BB \\
\CS - \CUK &= \Split - \UK
\end{align*}

\subsection{Species identities involving colored split graphs} \label{sec:identities}

This section includes three very bijective proofs of species identities involving colored split graphs. The colored split graphs are needed so that we can get the results on split graphs that we really want, in the next section.

The first theorem is a partial analog of \eqref{eqn:identity} on the level of species:

\begin{theorem} \label{thm:uk}
$\UK = E_{\ge2} \cdot \CS$.
\end{theorem}

\begin{proof}
We find a bijection between the labeled structures that commutes with isomorphisms, i.e.\ that is invariant under permuting the labels. The left side, $\UK$, counts $K$-canonical split graphs. The right side, $E_{\ge2} \cdot \CS$, counts ordered pairs of a set of size $\ge 2$ and a colored split graph.

Let $G$ be a labeled $K$-canonical split graph, with its canonical $K$-max partition. Color the vertices in $K$ green and the vertices in $S$ red. Coloring the structures of $\UK$ in this way does not change $\UK$, because isomorphisms between $K$-canonical split graphs preserve the $K$-max partition.

We now define the bijection by mapping $G$ to $(A,G - A)$, where $A$ is the set of swing vertices of $G$ and $G - A$ is the colored graph obtained by removing $A$ from $G$. The swing vertices form a clique of size $\ge 2$, so $A$ is a structure in $E_{\ge2}$. The green set and the red set in $G-A$ (inheriting the colors from $G$) are respectively $K$ and $S$ in a $KS$-partition of $G-A$; so, to show that $G-A$ is a colored split graph, we show that this partition is $S$-max.

In the canonical $K$-max partition of $G$, all the swing vertices of $G$ are in $K$, so no vertices of $S$ are removed from $G$ to form $G-A$. Suppose $G-A$ is not $S$-max. Then, by Proposition \ref{prop:CCT}, $G-A$ has a vertex $v \in K$ adjacent to none of the vertices in $S$. But $G$ and $G-A$ have the same $S$, so $v$ can be moved to $S$ to form a new $KS$-partition of $G$, making $v$ a swing vertex of $G$. Thus $v \in A$, a contradiction. Therefore the green set and the red set form an $S$-max partition of $G-A$.

The mapping $G \mapsto (A, G-A)$ does not depend on how the labels of the vertices in $G$ are permuted, precisely because the $K$-max partition of $G$ is unique.

To show this is a bijection, we describe its inverse. Let $A$ be a set of size $\ge 2$ and let $H$ be a colored split graph with chosen partition $K \cup S$, with the elements of $A$ and the vertices of $H$ given distinct labels from a shared label set. We obtain a split graph $G$ from the ordered pair $(A, H)$ by adding the elements of $A$ into $K$ as follows: put an edge between every element of $A$ and every vertex in $K$, and also put an edge between every pair of elements of $A$. Every element of $A$ is now a swing vertex, so the swing vertices of $G$ form a clique of size $\ge 2$, and so $G$ is $K$-canonical.

The set $A$ in the mapping $(A, H) \mapsto G$ becomes the set of swing vertices of $G$; and conversely, in the canonical partition of a $K$-canonical graph, the swing vertices are all in $K$ and are adjacent to no vertices in $S$ (by Proposition \ref{prop:CCT}). Hence these two functions are inverses, proving that $\UK = E_{\ge2} \cdot \CS$.
\end{proof}

Now that we have done a detailed proof of species equality, the next ones will be somewhat abbreviated, as they use the same idea.

\begin{theorem} \label{thm:uamb}
$\Uamb = X \cdot \BB$.
\end{theorem}

\begin{proof}
The left side, $\Uamb$, counts ambiguous split graphs. The right side, $X \cdot \BB$, counts ordered pairs of a single element and a balanced split graph. We go from the left side to the right side as follows: given an ambiguous split graph $G$ with swing vertex $a$, map $G$ to $(a, G - a)$. It turns out that $G - a$ is balanced, which we prove below.

The inverse, going from the right side to the left side, is as follows: given a single element $a$ and a balanced split graph $H$, append $a$ as a new swing vertex in $H$, adding an edge between it and every vertex in $K$. This is well-defined precisely because a balanced split graph has a unique $KS$-partition.

We now need to show that, if $G$ is an ambiguous split graph with swing vertex $a$, then $G-a$ is balanced. By Proposition \ref{prop:CCT}, the vertices of $G - a$ can be partitioned into a clique $Y$ and a stable set $Z$ such that $a$ is adjacent to everything in $Y$ and nothing in $Z$. Suppose $G-a$ is unbalanced. Without loss of generality, the partition $Y \cup Z$ is a $K$-max partition of $G-a$. Then by Proposition \ref{prop:CCT} there is $y \in Y$ that is adjacent to nothing in $Z$. This makes $y$ a swing vertex of $G$ as well, contradicting that $a$ is the only swing vertex of $G$. Therefore, $G-a$ is a balanced split graph, as claimed.
\end{proof}

\begin{theorem} \label{thm:cuk}
$\CUK = X \cdot E_{\ge1} \cdot \CS$.
\end{theorem}

\begin{proof}
We can think of $X \cdot E_{\ge1}$ as the species of ``pointed sets'' of size $\ge 2$, a pointed set being a set with a chosen distinguished element. The left side, $\CUK$, counts colored $K$-canonical split graphs. The right side, $X \cdot E_{\ge1} \cdot \CS$, counts ordered pairs of a pointed set of size $\ge 2$ and a colored split graph.

A colored $K$-canonical split graph can be obtained by taking a $K$-canonical split graph and choosing one of its swing vertices to be in $S$. From Theorem \ref{thm:uk} we have
\[ \UK = E_{\ge2} \cdot \CS, \]
where $E_{\ge2}$ represents the set of swing vertices in a $K$-canonical split graph; choosing one of the swing vertices to be in $S$ can be accomplished by replacing $E_{\ge2}$ with the species of pointed sets $X \cdot E_{\ge1}$, which yields the desired result.
\end{proof}

We summarize the species equalities we obtained in this section:
\begin{align*}
\UK &= E_{\ge2} \cdot \CS \\
\Uamb &= X \cdot \BB \\
\CUK &= X \cdot E_{\ge1} \cdot \CS
\end{align*}

\subsection{Main theorem on the species of split graphs} \label{sec:maintheorem}

In this section, we will manipulate the species equalities found in the previous two sections, obtaining an equation relating the species of unbalanced split graphs and the species of split graphs. This is where the ring of virtual species finally pays off.

\begin{theorem} \label{thm:main}
$\UU = \frac{(2-X)\cdot E - 2}{(1-X) \cdot E} \cdot \Split$.
\end{theorem}

\begin{proof}
\begin{align*}
\CS &= \Split - \UK + \CUK & \text{(Prop.\ \ref{prop:colored-equalities})} \\
\CS &= \Split - \UK + X \cdot E_{\ge1} \cdot \CS & \text{(Thm.\ \ref{thm:cuk})},
\end{align*}
and we can solve for $\CS$ (in the ring of virtual species) to get
\begin{equation} \label{eqn:cs} \CS = \frac{\Split - \UK}{1 - X \cdot E_{\ge1}}. \end{equation}
The virtual species $1 - X \cdot E_{\ge1}$ has constant term $1$, so it is a unit and we can divide by it.

Now we can find $\UK$ in terms of $\Split$ alone:
\begin{align*}
\UK &= E_{\ge2} \cdot \CS & \text{(Thm.\ \ref{thm:uk})} \\
\UK &= E_{\ge2} \cdot \frac{\Split - \UK}{1 - X \cdot E_{\ge1}} & \text{(Eqn.\ (\ref{eqn:cs}))},
\end{align*}
and we can solve for $\UK$ to get
\[ \UK = \frac{E_{\ge2}}{1 - X \cdot E_{\ge1} + E_{\ge2}} \cdot \Split. \]
By substituting $E_{\ge1} = E - 1$ and $E_{\ge2} = E - 1 - X$, we get
\begin{equation} \label{eqn:uk} \UK = \frac{E - 1 - X}{(1-X)\cdot E} \cdot \Split. \end{equation}
The species $(1+X)\cdot E$ has constant term $1$, so it is a unit and we can divide by it.

Now we can find $\UU$ in terms of $\Split$ alone:
\begin{align*}
\UU &= \UK + \US + \Uamb \\
&= 2\,\UK + \Uamb \\
&= 2\,\UK + X \cdot \BB & \text{(Thm.\ \ref{thm:uamb})} \\
&= 2\,\UK + X \cdot (\Split - \UU) \\
\UU &= 2\,\frac{E - 1 - X}{(1-X)\cdot E} \cdot \Split + X \cdot (\Split - \UU) & \text{(Eqn.\ (\ref{eqn:uk}))},
\end{align*}
and solving for $\UU$ gives the desired theorem.
\end{proof}

Theorem \ref{thm:main} is the true generalization of \eqref{eqn:identity}. It expresses the species of unbalanced split graphs as the product of split graphs with a species involving sets. In fact, \eqref{eqn:identity} is recovered from Theorem \ref{thm:main} by passing to the unlabeled generating functions, using the fact that $\widetilde{E}(x) = \frac{1}{1-x}$. In the same way, we can pass to the labeled generating functions, using the fact that $E(x) = e^x$. This process yields a new result:

\begin{theorem} \label{thm:labeled-U-S}
$\UU(x) = \frac{(2-x)e^x - 2}{(1-x)e^x}\,\Split(x) = \frac{2 - x - 2e^{-x}}{1-x}\,\Split(x)$. \hfill $\square$
\end{theorem}

\section{Bicolored graphs and asymptotics} \label{sec:BicoloredAsymptotics}

This section concerns asymptotic enumeration. Given non-negative sequences $x_n$ and $y_n$, we say $x_n$ asymptotically equals $y_n$ if $\lim_{n\to\infty} \frac{x_n}{y_n} = 1$, and we write this as $x_n \sim y_n$. If $x_n$ counts certain objects of size $n$, then we say almost all of the objects have a certain property if the fraction of size-$n$ objects with that property goes to $1$ as $n \to \infty$.

A \emph{bicolored graph} is a graph in which each vertex is colored green or red such that no two adjacent vertices are the same color. In other words, it is a bipartite graph with a chosen bipartition. Green and red are not interchangeable, meaning that swapping the color of every vertex will generally result in a different bicolored graph; that is, the two parts in the chosen bipartition are an \emph{ordered} pair. Bicolored graphs are often studied as a step towards bipartite graphs, as in \cite{Hanlon, GDG}.

Let $\BC$ be the species of bicolored graphs, where the isomorphisms between structures are the graph isomorphisms that preserve color. The species $\BC$ is fundamental, in the sense that its cycle index series and other associated generating functions have explicit formulas that can be derived from scratch rather than built up from those of simpler species (see \cite{GDG}). In particular, the number of labeled bicolored graphs is simple to express and routine to derive:
\begin{equation} |\BC[n]| = \sum_{k=0}^n \binom{n}{k} 2^{k(n-k)}, \label{eqn:bicolored-formula} \end{equation}
because we choose a subset of $k$ vertices to color green, and then for each of the $k(n-k)$ pairs of opposite-color vertices we choose whether to make them adjacent.

In this section, we prove some identities about the species of split graphs and the species of bicolored graphs (Section \ref{sec:relating-split-bicolored}); and we prove some asymptotic results on the number of split graphs and unbalanced split graphs (Sections \ref{sec:almost-all-labeled} and \ref{sec:almost-all-unlabeled}).

\subsection{The species of bicolored graphs} \label{sec:relating-split-bicolored}

Let $\BCiso$ be the species of bicolored graphs in which no green vertex is isolated. Then clearly $\BC = E \cdot \BCiso$, because $E$ provides the set of isolated green vertices. Collins and Trenk \cite{CollinsTrenk} find a bijection between unlabeled split graphs and unlabeled $\BCiso$-graphs (they say ``$XY$-graph'' where we say ``bicolored graph''), and their bijection easily proves this species identity:

\begin{theorem} \label{thm:CS-BCiso}
$\CS = \BCiso = \frac{\BC}{E}$.
\end{theorem}

\begin{proof}
The equality $\BCiso = \frac{\BC}{E}$ is clear (see the paragraph preceding the theorem), and note that $E$ is a unit in the ring of virtual species.

Now we define a bijection between labeled colored split graphs and bicolored graphs with no isolated green vertex: given a colored split graph, remove the edges in the green clique. The inverse is: given a bicolored graph with no isolated green vertex, add an edge between every pair of green vertices, making the green vertices a clique. A swing vertex in $K$ in the colored split graph becomes an isolated green vertex in the bicolored graph, and the converse is also true; consequently the bijection does take colored split graphs to $\BCiso$-graphs, and likewise for the inverse map. Since this bijection respects isomorphism of colored graphs, the equality of species is proved.
\end{proof}

Using Theorem \ref{thm:CS-BCiso} and some of the species identities from Sections \ref{sec:identities} and \ref{sec:maintheorem}, we obtain a relationship between $\Split$ and $\BC$:

\begin{theorem} \label{thm:bc-split}
$\BC = \frac{1}{1-X} \cdot \Split$.
\end{theorem}

\begin{proof}
By Theorem \ref{thm:uk}, $\UK = E_{\ge2} \cdot \CS$. By Theorem \ref{thm:CS-BCiso}, $\CS = \frac{\BC}{E}$. Therefore,
\begin{equation} \UK = \frac{E_{\ge2}}{E} \cdot \BC. \label{eqn:uk-bc} \end{equation}
Equation \eqref{eqn:uk} is that 
\[ \UK = \frac{E_{\ge2}}{(1-X)\cdot E} \cdot \Split; \]
setting this equal to (\ref{eqn:uk-bc}) and canceling the factor of $\frac{E_{\ge2}}{E}$ yields the desired result.
\end{proof}

Note that $\frac{1}{1-X}$ is the species of sequences, or linear orders. Is there a direct combinatorial way in which a bicolored graph is a split graph and a sequence?

By passing to the labeled and unlabeled generating functions, we obtain:

\begin{corollary} \label{cor:BC-Split-gf}
$\BC(x) = \frac{1}{1-x}\,\Split(x)$ and $\widetilde{\BC}(x) = \frac{1}{1-x}\,\widetilde{\Split}(x)$. \hfill $\square$
\end{corollary}

The latter equation was already known, as it comes from $\widetilde{U}(x) = x\,\widetilde{\BC}(x)$ \cite[Thm.\ 15]{CollinsTrenk} and $\widetilde{U}(x) = \frac{x}{1-x}\,\widetilde{\Split}(x)$ (which is \eqref{eqn:identity}). The former equation, when rearranged, implies that the number of labeled split graphs on $n$ vertices is $|\BC[n]| - n\,|\BC[n-1]|$; as we have an explicit formula for $|\BC[n]|$ in \eqref{eqn:bicolored-formula}, we obtain:

\begin{corollary} \label{cor:Bina}
The number of labeled split graphs on $n$ vertices is
\[ \sum_{k=0}^n \binom{n}{k} 2^{k(n-k)} - n\sum_{k=0}^{n-1} \binom{n-1}{k} 2^{k(n-k-1)}. \tag*{$\square$} \]
\end{corollary}

This is a different formula than the one obtained by B\'ina and P\v{r}ibil \cite{BP}; theirs gives the number of labeled split graphs as
\[ \label{eqn:Bina} 1 + \sum_{k=2}^n \binom{n}{k} \left( (2^k-1)^{n-k} - \sum_{j=1}^{n-k} \frac{jk}{j+1} \binom{n-k}{j} (2^{k-1}-1)^{n-k-j} \right). \]
Apparently this gives the same numbers as Corollary \ref{cor:Bina}; this of course is assured by the validity of both their proof and ours, and we have also checked by computer that the two agree up to $n=318$ (the $n$ that had been reached by the time we finished our sandwich). It would probably not be too hard to find an elementary proof that they are equal, by rearranging the sums and using properties of binomial coefficients.

\subsection{Almost all labeled split graphs are balanced} \label{sec:almost-all-labeled}

Let $b_n$ be the number of labeled bicolored graphs on $n$ vertices, let $s_n$ be the number of labeled split graphs on $n$ vertices, and let $u_n$ be the number of labeled unbalanced split graphs on $n$ vertices. In this subsection we prove that almost all labeled split graphs are balanced: $\lim_{n\to\infty} \frac{s_n - u_n}{s_n} = 1$, or equivalently $\lim_{n\to\infty} \frac{u_n}{s_n} = 0$.

Bender, Richmond, and Wormald \cite{BRW} show that
\begin{equation} \label{eqn:BRW} s_n\, \sim\, b_n\, =\, \sum_{k=0}^n \binom{n}{k} 2^{k(n-k)}\, \sim\, c(n) \, \binom{n}{\lfloor n/2 \rfloor} \,2^{n^2/4} \end{equation}
where $c(n) > 0$ depends only on whether $n$ is even or odd --- in particular,
\[ c(n) = \begin{cases}
\sum_{k\in\mathbb{Z}} 2^{-k^2} \approx 2.128937 & \text{if $n$ is even;} \\
\sum_{k\in\mathbb{Z}} 2^{-(k+\frac{1}{2})^2} \approx 2.128931 & \text{if $n$ is odd.}
\end{cases} \]
Also note that the equality $b_n = \sum_{k=0}^n \binom{n}{k} 2^{k(n-k)}$ in \eqref{eqn:BRW} is just \eqref{eqn:bicolored-formula}.

From \eqref{eqn:BRW} we immediately obtain this lemma:
\begin{lemma} \label{lem:b-ratio}
$\frac{s_n}{s_{n-1}} \ge 2^{(n+1)/2}$ and $\frac{b_n}{b_{n-1}} \ge 2^{(n+1)/2}$, for large enough $n$. \hfill $\square$
\end{lemma}

\begin{theorem} \label{thm:almost-all-balanced}
Almost all labeled split graphs are balanced: $\lim_{n\to\infty} \frac{u_n}{s_n} = 0$.
\end{theorem}

\begin{proof}
Let $A(x) = \frac{2 - x - 2e^{-x}}{1-x}$. From Theorem \ref{thm:labeled-U-S}, we have $\UU(x) = A(x)\,\Split(x)$. Let $a_i$ be the coefficient of $x^i$ in $A(x)$; then $0 \le a_i \le 1$ for all $i$, and $a_0 = 0$. Then
\[ \frac{u_n}{n!} = \sum_{k=0}^{n-1} a_{n-k} \frac{s_k}{k!} \le \sum_{k=0}^{n-1} \frac{s_k}{k!}. \]
By Lemma \ref{lem:b-ratio}, for large enough $n$, the highest term in this sum is the last one, $\frac{s_{n-1}}{(n-1)!}$. Thus, for large enough $n$,
\begin{align*}
\frac{u_n}{n!} &\le n \frac{s_{n-1}}{(n-1)!} \\
u_n &\le n^2 s_{n-1} \\
\frac{u_n}{s_n} &\le \frac{n^2 s_{n-1}}{s_n} \le \frac{n^2}{2^{(n+1)/2}}
\end{align*}
(this last inequality uses Lemma \ref{lem:b-ratio}), and $\frac{n^2}{2^{(n+1)/2}}$ goes to $0$ as $n \to \infty$.
\end{proof}

\subsection{Almost all unlabeled split graphs are balanced} \label{sec:almost-all-unlabeled}

Let $\widetilde{b}_n$ be the number of unlabeled bicolored graphs on $n$ vertices, let $\widetilde{s}_n$ be the number of unlabeled split graphs on $n$ vertices, and let $\widetilde{u}_n$ be the number of unlabeled unbalanced split graphs on $n$ vertices. We will use the following ``folklore'' result about bicolored graphs (proved e.g.\ in \cite{Wright}):

\begin{theorem}[see \cite{Wright}] \label{thm:unlabeled-bicolored-asymptotic}
The number of unlabeled bicolored graphs on $n$ vertices asymptotically equals the number of labeled bicolored graphs on $n$ vertices divided by $n!$. That is, $\widetilde{b}_n \sim b_n/n!$.
\end{theorem}

We will now use this to show that $\lim_{n\to\infty} \frac{\widetilde{u}_n}{\widetilde{s}_n} = 0$. Along the way, we also find that $\widetilde{s}_n \sim \widetilde{b}_n$.

Theorem \ref{thm:unlabeled-bicolored-asymptotic} allows us to convert our results on labeled graphs into results on unlabeled graphs. From \eqref{eqn:BRW}, we obtain:

\begin{lemma}
$\widetilde{b}_n \sim \frac{c(n)}{\lfloor n/2 \rfloor!\, \lceil n/2 \rceil!} \,2^{n^2/4}$, where $c(n)$ is defined as in Section \ref{sec:almost-all-labeled}. \hfill $\square$
\end{lemma}

And from Lemma \ref{lem:b-ratio}, we obtain:

\begin{lemma} \label{lem:b-ratio-unlabeled}
$\frac{\widetilde{b}_n}{\widetilde{b}_{n-1}} \ge 2^{(n+1)/2} / n$, for large enough $n$. \hfill $\square$
\end{lemma}

We now give a result on $\widetilde{s}_n$ analogous to \eqref{eqn:BRW}:

\begin{theorem} \label{thm:snsimbn-unlabeled}
The number of unlabeled split graphs on $n$ vertices asymptotically equals the number of unlabeled bicolored graphs on $n$ vertices: $\widetilde{s}_n \sim \widetilde{b}_n$.
\end{theorem}

\begin{proof}
By Corollary \ref{cor:BC-Split-gf}, $\widetilde{s}_n = \widetilde{b}_n - \widetilde{b}_{n-1}$. Then $\lim_{n\to\infty} \frac{\widetilde{s}_n}{\widetilde{b}_n} = 1 - \lim_{n\to\infty} \frac{\widetilde{b}_{n-1}}{\widetilde{b}_n}$, and Lemma \ref{lem:b-ratio-unlabeled} implies that $\lim_{n\to\infty}\frac{\widetilde{b}_{n-1}}{\widetilde{b}_n} = 0$.
\end{proof}

Theorem \ref{thm:snsimbn-unlabeled} has three immediate corollaries:

\begin{corollary}
The number of unlabeled split graphs on $n$ vertices asymptotically equals the number of labeled split graphs on $n$ vertices divided by $n!$. That is, $\widetilde{s}_n \sim s_n/n!$. \hfill $\square$
\end{corollary}

\begin{corollary}
$\widetilde{s}_n \sim \frac{c(n)}{\lfloor n/2 \rfloor!\, \lceil n/2 \rceil!} \,2^{n^2/4}$, where $c(n)$ is defined as in Section \ref{sec:almost-all-labeled}. \hfill $\square$
\end{corollary}

\begin{corollary} \label{cor:s-ratio-unlabeled}
$\frac{\widetilde{s}_n}{\widetilde{s}_{n-1}} \ge 2^{(n+1)/2} / n$, for large enough $n$. \hfill $\square$
\end{corollary}

We now finally reach the proof of the conjecture of Cheng, Collins, and Trenk \cite{CCT}:

\begin{theorem} \label{thm:almost-all-balanced-unlabeled}
Almost all unlabeled split graphs are balanced: $\lim_{n\to\infty} \frac{\widetilde{u}_n}{\widetilde{s}_n} = 0$.
\end{theorem}

\begin{proof}
By \eqref{eqn:identity}, $\widetilde{u}_n = \sum_{k=0}^{n-1} \widetilde{s}_k$. By Corollary \ref{cor:s-ratio-unlabeled}, for large enough $n$, the highest term in this sum is the last one, $\widetilde{s}_{n-1}$. Thus $\widetilde{u}_n \le n\,\widetilde{s}_{n-1}$, for large enough $n$, and so
\[ \frac{\widetilde{u}_n}{\widetilde{s}_n} \le \frac{n \widetilde{s}_{n-1}}{\widetilde{s}_n} \le \frac{n^2}{2^{(n+1)/2}} \]
(this last inequality uses Corollary \ref{cor:s-ratio-unlabeled}), and $\frac{n^2}{2^{(n+1)/2}}$ goes to $0$ as $n \to \infty$.
\end{proof}

\section*{Acknowledgements}

Thanks to Andrew Gainer-Dewar, for introducing me to combinatorial species and graph enumeration in 2013. Also thanks to Karen Collins: I found out about split graphs from her talk at Queens College in October 2017. Finally, thanks to Catherine Greenhill (editor for the Electronic Journal of Combinatorics) and Tomasz \L{}uczak, for tracking down the reference for Theorem \ref{thm:unlabeled-bicolored-asymptotic}.

\vspace{1pc}\renewcommand{\section}[2]{\textbf{References.}}
\footnotesize{

\end{document}
\begin{thebibliography}{10} \vspace{-0.35pc}

\bibitem{BRW} E. A. Bender, L. B. Richmond, and N. C. Wormald, Almost all chordal graphs split, \emph{J. Austral.\ Math.\ Soc.\ (Series A)} \textbf{38} (1985), 214--221.

\bibitem{BLL} F. Bergeron, G. Labelle, and P. Leroux, \emph{Combinatorial Species and Tree-like Structures}, trans.\ M. Readdy, Encyclopedia of Mathematics and its Applications, Cambridge University Press, Cambridge, 1998.

\bibitem{BP} V. B\'ina and J. P\v{r}ibil, Note on enumeration of labeled split graphs, \emph{Comment.\ Math.\ Univ.\ Carolin.}\ \textbf{56} (2015), 133--137.

\bibitem{CCT} C. Cheng, K. L. Collins, and A. N. Trenk, Split graphs and Nordhaus--Gaddum graphs, \emph{Discrete Math.}\ \textbf{339} (2016), 2345--2356.

\bibitem{CollinsTrenk} K. L. Collins and A. N. Trenk, Finding balance: Split graphs and related classes, \emph{Electron.\ J. Combin.}\ \textbf{25} (2018), \#P1.73.

\bibitem{GDG} A. Gainer-Dewar and I. M. Gessel, Enumeration of bipartite graphs and bipartite blocks, \emph{Electron.\ J. Combin.}\ \textbf{21} (2014), \#P2.40.

\bibitem{Golumbic} M. C. Golumbic, \emph{Algorithmic Graph Theory and Perfect Graphs}, Academic Press, New York, 1980.

\bibitem{HS} P. L. Hammer and B. Simeone, The splittance of a graph, \emph{Combinatorica} \textbf{1} (1981), 275--284.

\bibitem{Hanlon} P. Hanlon, The enumeration of bipartite graphs, \emph{Discrete Math.}\ \textbf{28} (1979), 49--57.


\bibitem{comps} A. Hardt, P. McNeely, T. Phan, and J. M. Troyka, Combinatorial species and graph enumeration, undergraduate thesis (2013), \href{http://arxiv.org/abs/1312.0542}{\texttt{arXiv:1312.0542}}.

\bibitem{Joyal} A. Joyal, Une th\'eorie combinatoire des s\'eries formelles, \emph{Adv.\ Math.}\ \textbf{42} (1981), 1--82.

\bibitem{Royle} G. F. Royle, Counting set covers and split graphs, \emph{J. Integer Seq.}\ \textbf{3} (2000), Article 00.2.6.

\bibitem{Wright} E. M. Wright, The $k$-connectedness of bipartite graphs, \emph{J. Lond.\ Math.\ Soc.\ (2)} \textbf{25} (1982), 7--12.

\end{thebibliography}
